\documentclass[12pt,reqno]{amsart}
\usepackage{enumerate, latexsym, amsmath, amsfonts, amssymb, amsthm, graphicx, color}
\def\pmod #1{\ ({\rm{mod}}\ #1)}
\def\Z{\Bbb Z}

\def\Q{\Bbb Q}

\def\bg{\bigg}
\def\({\bg(}
\def\){\bg)}

\def\sgn{{\rm sgn}}

\def\sgn{{\rm sgn}}
\def\Arg{{\rm Arg}}
\def\Norm{{\rm Norm}}
\def\Gal{{\rm Gal}}

\def\ve{\varepsilon}

\theoremstyle{plain}
\newtheorem{theorem}{Theorem}

\newtheorem{lemma}{Lemma}
\newtheorem{corollary}{Corollary}

\theoremstyle{definition}
\newtheorem*{acknowledgment}{Acknowledgment}
\theoremstyle{remark}
\newtheorem{remark}{Remark}

\makeatletter
\@namedef{subjclassname@2010}{%
  \textup{2010} Mathematics Subject Classification}
\makeatother
 \vspace{4mm}

\begin{document}
 \baselineskip=17pt
\hbox{} {}
\medskip
\title[On a polynomial involving roots of unity and its applications]
{On a polynomial involving roots of unity and its applications}
\date{}
\author[Hai-Liang Wu and Yue-Feng She] {Hai-Liang Wu and Yue-Feng She}

\thanks{
\newline \indent Supported by the National Natural Science
Foundation of China (Grant No. 11971222).}

\address {(Hai-Liang Wu) Department of Mathematics, Nanjing
University, Nanjing 210093, People's Republic of China}
\email{{\tt whl.math@smail.nju.edu.cn}}

\address {(Yue-Feng She) Department of Mathematics, Nanjing
University, Nanjing 210093, People's Republic of China}
\email{{\tt she.math@smail.nju.edu.cn}}

\begin{abstract}
Let $p>3$ be a prime. Gauss first introduced the polynomial $S_p(x)=\prod_{c}(x-\zeta_p^c),$
where $0<c<p$ and $c$ varies over all quadratic residues modulo $p$ and $\zeta_p=e^{2\pi i/p}$. Later Dirichlet investigated this polynomial and used this to solve the problems involving the Pell equations. Recently, Z.-W Sun studied some trigonometric identities involving this polynomial. In this paper, we generalized their results. As applications of our result, we extend S. Chowla's result on the congruence concerning the fundamental unit of $\mathbb{Q}(\sqrt{p})$ and give an equivalent form of the
extended Ankeny-Artin-Chowla conjecture.
\end{abstract}

\maketitle

\section{Introduction}
\setcounter{lemma}{0}
\setcounter{theorem}{0}
\setcounter{corollary}{0}
\setcounter{remark}{0}
\setcounter{equation}{0}
\setcounter{conjecture}{0}
Let $p>3$ be a prime, and let $\zeta_p=e^{2\pi i/p}$.
In his outstanding book {\it Disquisitiones Arithmeticae} Gauss first introduced the following two polynomials.
\begin{align*}
S_p(x)&:=\prod_{c}(x-\zeta_p^c),\\
T_p(x)&:=\prod_{b}(x-\zeta_p^b),
\end{align*}
where $c$ varies over all quadratic residues modulo $p$ in the interval $[1,p]$ and $b$ runs over all quadratic non-residues modulo $p$ in the interval $[1,p]$. Gauss proved that there are certain polynomials $Y_p(x)$,$Z_p(x)$ having rational integral coefficients such that
$$2S_p(x)=Y_p(x)+\sqrt{p^*}Z_p(x),\ 2T_p(x)=Y_p(x)-\sqrt{p^*}Z_p(x),$$
where $p^*=(-1)^{\frac{p-1}{2}}p.$ And he also showed that
$$4\Phi_p(x)=Y_p(x)^2-p^*Z_p(x)^2,$$
where $\Phi_p(x)$ is the $p$-th cyclotomic polynomial. In 1811 Gauss applied these polynomials in his determination of the sign of the Gaussian sum. Later
Dirichlet observed that if $p\equiv3\pmod4$ and $p>3$, then we have
\begin{equation}\label{equation Dirichlet}
S_p(i)=\frac{1}{2}m_p(1+(-1)^{(p+1)/4}i)+\frac{1}{2}n_p(1-(-1)^{(p+1)/4}i)\sqrt{-p}
\end{equation}
with some rational integer $m_p,n_p$ (readers may consult \cite[p. 371]{D} for more details).
However, Dirichlet did not give the explicit value of $S_p(i)$.

Recently, Z.-W Sun \cite{Sun2} investigated many trigonometric identities. In particular, when $p\equiv1\pmod 4$ be a prime, he obtained the value of $S_p(i)$. In the case $p\equiv3\pmod4$, Z.-W Sun
also posed a conjecture concerning the explicit value of $S_p(i)$. More precisely, 
let $\ve_{4p}>1$ be the fundamental unit of $\Q(\sqrt{p})$. In \cite[Conjecture 5.1]{Sun2} he conjectured that if we let $\ve_{4p}^{h(4p)}=A_p+B_p\sqrt{p}$ with $A_p,B_p\in\Z$, then we have 
$$(i-(-1)^{\frac{p+1}{4}})S_p(i)=(-1)^{\frac{h(-p)+1}{2}\cdot\frac{p+1}{4}}(s_p-t_p\sqrt{p}),$$
where $s_p=\sqrt{A_p+(-1)^{\frac{p+1}{4}}}$, $t_p=B_p/s_p$ and $h(-p)$ denotes the class number of $\Q(\sqrt{-p})$. 

For all positive odd integers $n$, we let $(\frac{\cdot}{n})$ be the Jacobi symbol and let $\zeta_n=e^{2\pi i/n}$. Motivated by the above works, we introduce the following polynomial:
$$S_n(x)=\prod_{c}(x-\zeta_n^c),$$
where $c$ varies over all integers in the interval $[1,n]$ satisfying $(\frac{c}{n})=1$. Let $n>3$ be an arbitrary squarefree integer with $n\equiv3\pmod4$. In this paper, we first determine the explicit value of $S_n(i)$ completely.

To state our first result, we first introduce some notations.
The symbol $\#S$ denotes the cardinality of a finite set $S$. Let
$n>3$ be an arbitrary squarefree integer with $n\equiv3\pmod4$. We first set
$$\alpha(n):=\#\{0<c<\frac{n}{8}:\ \(\frac{c}{n}\)=1\}\cup\{\frac{5n}{8}<c<n:\ \(\frac{c}{n}\)=1\}.$$
In addition, if $n=p$ is a prime, then we let
$$\beta(p):=1+\#\{\frac{p}{8}<c<\frac{3p}{8}: \(\frac{c}{p}\)=-1\}.$$
It might worth mentioning here that when $p\equiv7\pmod8$, then we have 
$$\beta(p)=1+\left\lfloor\frac{p}{8}\right\rfloor-\frac{h(-8p)}{4},$$
where $\lfloor\cdot\rfloor$ is the floor function and $h(-8p)$ is the class number of $\Q(\sqrt{-2p})$. 
In fact, by \cite{HW} we have the following formula 
$$h(-8p)=2\sum_{p/8<x<3p/8}\(\frac{x}{p}\).$$
From this one can easily verify that 
$$\beta(p)=1+\left\lfloor\frac{p}{8}\right\rfloor-\frac{h(-8p)}{4}.$$
Throughout this paper, for an arbitrary real quadratic field $\Q(\sqrt{D})$ of discriminant $D$, we let
$\ve_D>1$ and $h(D)$ be the fundamental unit and the class number of $\Q(\sqrt{D})$ respectively.
As usual, the Euler totient function is denoted by $\phi$. Now we are in the position to state our first result.
\begin{theorem}\label{Theorem A}
Let $n>3$ be a squarefree integer with $n\equiv3\pmod4$. Then we have

{\rm (i)} If $n$ is not a prime, then we have
$$S_n(i)=(-1)^{\frac{\phi(n)}{8}+\alpha(n)}\cdot\ve_{4n}^{-h(4n)/2}.$$
In this case, we can write
$$S_n(i)=a_n+b_n\sqrt{n}$$
for some $a_n,b_n\in\Z$. And $(a_n,b_n)$ is a solution of the Pell equation
$$x^2-ny^2=1\ (xy\ne0).$$

{\rm (ii)} If $n=p$ is a prime, then we have
$$S_p(i)=(-1)^{\beta(p)}\cdot\ve_{4p}^{-h(4p)/2}\cdot(1+i(-1)^{(p+1)/4})/\sqrt{2}.$$
In this case, we can write
$$(i-(-1)^{(p+1)/4})S_p(i)=a_p+b_p\sqrt{p}$$
for some $a_p,b_p\in\Z$. And $(a_p,b_p)$ is a solution of the equation
$$x^2-py^2=\(\frac{2}{p}\)2.$$
\end{theorem}

To make our result more explicit, we give the following examples.

(i) When $n=15$, we have $h(60)=2$, $\ve_{60}=4+\sqrt{15}$, $\alpha(15)=1$ and $\phi(15)=8$. With the help of computer we obtain
$$S_{15}(i)=4-\sqrt{15}=(-1)^{\phi(15)/8+J(15)}\cdot\ve_{60}^{-h(60)/2}.$$

(ii) When $n=7$, we have $h(28)=1$, $\ve_{28}=8+3\sqrt{7}$ and $\beta(7)=1$. By computation we obtain
$$S_7(i)=\frac{-3+\sqrt{7}-i\sqrt{16-6\sqrt{7}}}{2},$$
and
$$(i-1)S_7(i)=\sqrt{\frac{2}{8+3\sqrt{7}}}=3-\sqrt{7}.$$

(iii) When $n=11$, we have $h(44)=1$, $\ve_{44}=10+3\sqrt{11}$ and $\beta(11)=2$. Via computation we obtain
$$S_{11}(i)=\frac{-3+\sqrt{11}-i\sqrt{20-6\sqrt{11}}}{2},$$
and
$$(i+1)S_{11}(i)=\sqrt{\frac{2}{10+3\sqrt{11}}}=-3+\sqrt{11}.$$
Now we introduce our second result. Mordell \cite{M} proved that if $p>3$ is a prime and $p\equiv 3\pmod4$ then
$$\(\frac{p-1}{2}\)!\equiv (-1)^{\frac{h(-p)+1}{2}}\pmod p,$$
where $h(-p)$ denotes the class number of $\Q(\sqrt{-p})$. Later S. Chowla \cite{Chowla} extended Mordell's result. Let $p\equiv1\pmod4$ be a prime, and let $\ve_{p}=(u_p+v_p\sqrt{p})/2>1$ be the fundamental unit of $\Q(\sqrt{p})$. Then S. Chowla showed that
$$\(\frac{p-1}{2}\)!\equiv \frac{(-1)^{(h(p)+1)/2}u_p}{2}\pmod p.$$

In view of the above, let $p>3$ be a prime with $p\equiv3\pmod4$, and let $\ve_{4p}=u_p+v_p\sqrt{p}>1$ be the fundamental unit of $\Q(\sqrt{p})$.
It is natural to investigate $u_p\mod p$. We obtain the following result which extends S. Chowla's result.

\begin{corollary}\label{Corollary A}
Let $p>3$ be a prime with $p\equiv3\pmod4$, and let $\ve_{4p}=u_p+v_p\sqrt{p}>1$ be the fundamental unit of $\Q(\sqrt{p})$. Then we have
$$u_p\equiv (-1)^{\frac{p+1}{4}}\pmod p.$$
\end{corollary}

For example, when $p=419$ the fundamental unit of $\Q(\sqrt{419})$ is given by
$$\ve_{4p}=270174970+13198911\sqrt{419}.$$
It is easy to verify that $270174970\equiv -1\pmod {419}.$

Ankeny, Artin, and Chowla \cite{AACC} posed the following interesting conjecture.

{\bf The Ankeny-Artin-Chowla conjecture}.\ Let $p\equiv1\pmod4$ be a prime, and let $\ve_p=\frac{u_p+v_p\sqrt{p}}{2}>1$ be the fundamental unit of $\Q(\sqrt{p})$. Then $p\nmid v_p$.

Later a more general conjecture
was suggested by Kiselev-Slavutskii \cite{AA} and independently by L. J. Mordell \cite{Modell2}.

{\bf The Extended Ankeny-Artin-Chowla conjecture}.\ If the squarefree part of the fundamental
discriminant $D$ of a real quadratic field is an odd prime $p$, then for the fundamental unit
$\ve_D=\alpha+\beta\sqrt{p}>1$ of this field we have $p\nmid \beta$.

There are many equivalent forms of the above conjectures. For example, it is well known that when $p\equiv1\pmod4$ is a prime,
the Ankeny-Artin-Chowla conjecture holds if and only if
$$B_{\frac{p-1}{2}}\not\equiv0\pmod p,$$
where $B_{\frac{p-1}{2}}$ is the $\frac{p-1}{2}$-th Bernoulli number.

With the help of the above results we can write
$(i-(-1)^{\frac{p+1}{4}})S_p(i)=a_p+b_p\sqrt{p}$ for some $a_p,b_p\in\Z$ if $p>3$ with $p\equiv3\pmod4$ is a prime. Now we can obtain the following result.
\begin{corollary}\label{corollary of EAACC}
Let $p>3$ be a prime with $p\equiv3\mod4$.
The extended Ankeny-Artin-Chowla conjecture holds for the real quadratic field $\Q(\sqrt{p})$ if and only if $p\nmid b_p$.
\end{corollary}

The proofs of the above results will be given in Section 2.
\maketitle
\section{Proofs of the main results}
\setcounter{lemma}{0}
\setcounter{theorem}{0}
\setcounter{corollary}{0}
\setcounter{remark}{0}
\setcounter{equation}{0}
\setcounter{conjecture}{0}
In this section, for each positive integer $m$ we use the symbol $\zeta_m$ to denote the number $e^{2\pi i/m}$. We begin with the following well known result concerning the class number of imaginary quadratic field.
\begin{lemma}\label{Lemma of class number of imaginary quadratic fields}
Let $n>3$ be a squarefree integer with $n\equiv3\pmod4$, and let $h(-n)$ be the class number of $\Q(\sqrt{-n})$. Then we have
$$h(-n)=\sum_{0<x<n/2}\(\frac{x}{n}\)-\frac{2}{n}\sum_{0<x<n/2}\(\frac{x}{n}\)x.$$
And hence we have
$$\sum_{0<x<n,(\frac{x}{n})=+1}x\equiv\sum_{0<x<n/2}\(\frac{x}{n}\)x\equiv0\pmod n.$$
\end{lemma}

\begin{proof}
By the class number formula of imaginary quadratic field (cf. \cite[p. 344, Theorem 1]{BS}) we have
\begin{align*}
-nh(-n)=&\sum_{0<x<n,(x,n)=1}\(\frac{-n}{x}\)x
\\=&\sum_{0<x<n/2,(x,n)=1}\(\frac{-n}{x}\)x+\sum_{0<x<n/2,(x,n)=1}\(\frac{-n}{n-x}\)(n-x)
\\=&\sum_{0<x<n/2,(x,n)=1}\(\frac{-n}{x}\)(2x-n),
\end{align*}
where $(\frac{-n}{\cdot})$ is the Kronecker symbol.
Note that when $n>0$ and $n\equiv3\pmod4$, for any positive integer $x$ with $(x,n)=1$ we have
$$\(\frac{-n}{x}\)=\(\frac{x}{n}\).$$
Hence we can easily get the desired formula of $h(-n)$. Moreover, the second result of the Lemma follows from the identity
$$\sum_{0<x<n,(\frac{x}{n})=+1}x=\sum_{0<x<n/2,(\frac{x}{n})=+1}x+\sum_{0<x<n/2,(\frac{x}{n})=-1}(n-x).$$
This completes the proof.
\end{proof}
\begin{remark}\label{remark of the class number formula of imaginary quadratic field}
Let notations be as in Lemma \ref{Lemma of class number of imaginary quadratic fields}. By \cite[p. 344, Theorem 3]{BS} we have
\begin{equation}\label{equation of another type of class number of imaginary field}
h(-n)=\frac{1}{2-(-1)^{(n+1)/4}}\sum_{0<c<n/2}\(\frac{c}{n}\)\equiv \frac{\phi(n)}{2}\pmod2.
\end{equation}
Hence if $n=p>3$ is an odd prime with $n\equiv3\pmod4$, then $h(-n)$ is odd.
If the squarefree integer $n>3$ with $n\equiv3\pmod4$ is not a prime,  then $\phi(n)\equiv0\pmod8$ and hence it is easy to see that $h(-n)$ is even.
\end{remark}

In 2004 R. Chapman \cite[Lemma 4]{R} showed that if $p>3$ is a prime with $p\equiv3\pmod4$, then
$$\prod_{0<c<p/2}(1-\zeta_p^{c^2})=(-1)^{\frac{h(-p)+1}{2}}\sqrt{-p}.$$
Later Z.-W Sun \cite{Sun1} obtained many identities involving $p$-th roots of unity with $p$ prime. Motivated by their works, we get the following result.
\begin{lemma}\label{Lemma of Sn(1)}
Let $n>3$ be a squarefree integer with $n\equiv 3\pmod4$, and let $h(-n)$ denote the class number of $\Q(\sqrt{-n})$.
Then for each integer $a$ prime to $n$ we have
$$\prod_{0<c<n, (\frac{c}{n})=+1}(1-\zeta_n^{ac})=
\begin{cases}(-1)^{\frac{h(-n)+1}{2}}(\frac{a}{n})\sqrt{-n}&\mbox{if}\ n\ \text{is prime},
\\(-1)^{\frac{h(-n)}{2}}&\mbox{otherwise}.\end{cases}$$
\end{lemma}
\begin{proof}
When $n>3$ is a prime, the above result is known (cf. \cite[Theorem 1.3]{Sun1}). Suppose now that
$n>3$ is not a prime. Let $\Phi_n(x)$ denote the $n$-th cyclotomic polynomial. It is known that
(cf. \cite[p. 142, Proposition 3.5.4]{Cohn})
\begin{equation}\label{equation of cyclotomic polynomial}
\Phi_n(1)=\prod_{0<c<n,(c,n)=1}(1-\zeta_n^c)=1.
\end{equation}
Applying the automorphism $\zeta_n\mapsto\zeta_n^{-1}$ we see that $S_n(1)$ and the product
$$\prod_{0<c<n,(\frac{c}{n})=-1}(1-\zeta_n^c)$$
are complex conjugated, and in view of (\ref{equation of cyclotomic polynomial}) and the equality
$$S_n(1)\prod_{0<c<n,(\frac{c}{n})=-1}(1-\zeta_n^c)=\left|S_n(1)\right|^2=1$$
we obtain $\left|S_n(1)\right|^2=1.$ Hence it suffices to determine the argument of this number.
We use the symbol $\Arg(z)$ to denote the argument of a complex number $z$. Let $\sigma_{-2}$ be an element in the Galois group $\Gal(\Q(\zeta_n)/\Q)$ with
$\sigma_{-2}(\zeta_n)=\zeta_n^{-2}$. We consider the number
$$\sigma_{-2}(S_n(1))=\prod_{0<c<n, (\frac{c}{n})=+1}(1-\zeta_n^{-2c}).$$
We have
\begin{align*}
\prod_{0<c<n, (\frac{c}{n})=+1}(1-\zeta_n^{-2c})
=&\prod_{0<c<n, (\frac{c}{n})=+1}\zeta_n^{-c}(\zeta_n^{c}-\zeta_n^{-c})
\\=&\prod_{0<c<n, (\frac{c}{n})=+1}\zeta_n^{-c}\cdot2i\cdot\sin\frac{2\pi c}{n}.
\end{align*}
Hence we have
\begin{equation*}
\Arg(\sigma_{-2}(S_n(1)))\equiv \sum_{0<c<n, (\frac{c}{n})=+1}\(\frac{-2\pi c}{n}
+\frac{\pi}{2}\sgn(\sin\frac{2\pi c}{n})\)\pmod{2\pi\Z},
\end{equation*}
where $\sgn(x)$ is the sign of a real number $x$. Via computation, we have
\begin{equation*}
\Arg(\sigma_{-2}(S_n(1)))\equiv \frac{-2\pi}{n}\sum_{0<c<n/2}\(\frac{c}{n}\)c
+\frac{\pi}{2}\sum_{0<c<n/2}\(\frac{c}{n}\)\pmod{2\pi\Z}.
\end{equation*}

With the help of Lemma \ref{Lemma of class number of imaginary quadratic fields} and (\ref{equation of another type of class number of imaginary field}), we finally get that
$$\Arg(\sigma_{-2}(S_n(1)))\equiv\frac{h(-n)}{2}\cdot\pi
\pmod{2\pi\Z}.$$
(Note that by Remark \ref{remark of the class number formula of imaginary quadratic field} we have
$2\mid h(-n)$ in this case.) Hence
$$\sigma_{-2}(S_n(1))=(-1)^{\frac{h(-n)}{2}}.$$
And
$$S_n(1)=\sigma_{-2}^2(S_n(1))=(-1)^{\frac{h(-n)}{2}}.$$
From this we can easily get the desired result.
\end{proof}

We next consider the product $S_n(i)S_n(-i)$.
\begin{lemma}\label{Lemma of Sn(i) times Sn(-i)}
Let $n>3$ be a squarefree integer with $n\equiv3\pmod4$. Then we have
$$S_n(i)S_n(-i)=\begin{cases}(-1)^{\frac{n+1}{4}}&\mbox{if}\ n\ \text{is prime},
\\1&\mbox{otherwise}.\end{cases}$$
\end{lemma}
\begin{proof}
\begin{equation*}
S_n(i)S_n(-i)=\prod_{0<c<n,(\frac{c}{n})=+1}(1+\zeta_n^{2c})
=\prod_{0<c<n,(\frac{c}{n})=+1}\frac{1-\zeta_n^{4c}}{1-\zeta_n^{2c}}.
\end{equation*}
Then the desired result follows from Lemma \ref{Lemma of Sn(1)}.
\end{proof}

Dirichlet first realized that $(i-(-1)^{\frac{p+1}{4}})S_p(i)\in\Z[\sqrt{p}]$ if $p>3$ is a prime with $p\equiv3\pmod4$. We extend Dirichlet's result and obtain the following result.

\begin{lemma}\label{Lemma of Galois}
Let $n>3$ be a squarefee integer with $n\equiv3\pmod4$. Then we have
$$\gamma(4n)S_n(i)\in\Z(\sqrt{n}),$$
where
$$\gamma(4n)=\begin{cases}i-(-1)^{\frac{n+1}{4}}&\mbox{if}\ n\ \text{is prime},
\\1&\mbox{otherwise}.\end{cases}$$
\end{lemma}
\begin{proof}
We divide our proof into two cases.

{\it Case}\ 1. $n=p$ is a prime.

In this case it is easy to see that the Galois group
$$\Gal\(\Q(i,\zeta_p)/\Q(\sqrt{p})\)=\{\tau_a: a\in(\Z/4p\Z)^{\times},\(\frac{p}{a}\)=+1 \},$$
where $\tau_a$ is a $\Q$-automorphism of $\Q(\zeta_{4p})$ defined by sending
$\zeta_{4p}$ to $\zeta_{4p}^a$.

For each $\tau_a\in\Gal(\Q(i,\zeta_p)/\Q(\sqrt{p}))$, clearly
$(\frac{p}{a})=+1$ implies that either $a\equiv 1\pmod 4$ and $(\frac{a}{p})=+1$ or $a\equiv3\pmod4$ and $(\frac{a}{p})=-1$. If $a\equiv 1\pmod 4$ and $(\frac{a}{p})=+1$, then $\tau_a$ acts trivially on $(i-(-1)^{\frac{p+1}{4}})S_p(i)$. Suppose now that $a\equiv3\pmod4$ and $(\frac{a}{p})=-1$. Then
$$\tau_a((i-(-1)^{\frac{p+1}{4}})S_p(i))=(-i-(-1)^{\frac{p+1}{4}})\prod_{1\le t\le \frac{p-1}{2}}(-i-\zeta_p^{-t^2}).$$
The equality
$$\prod_{1\le t\le\frac{p-1}{2}}(-i-\zeta_p^{t^2})\times\prod_{1\le t\le\frac{p-1}{2}}(-i-\zeta_p^{-t^2})
=\frac{(-i)^{p}-1}{-i-1}=-i.$$
And Lemma \ref{Lemma of Sn(i) times Sn(-i)} imply that
$$S_p(i)S_p(-i)=(-1)^{\frac{p+1}{4}}.$$
From the above we see that $\tau_a$ fixes $(i-(-1)^{\frac{p+1}{4}})S_p(i)$ and hence we have $(i-(-1)^{\frac{p+1}{4}})S_p(i)\in\Z[\sqrt{p}]$.

{\it Case.}\ 2. $n$ is not a prime.

Recall that $\Phi_n(x)$ is the $n$-th cyclotomic polynomial. We first show that in this case
\begin{equation}\label{equation of Phi n (i)}
\Phi_n(i)=1.
\end{equation}
In fact, we have
\begin{align*}
\Phi_n(i)=&\prod_{0<c<n,(\frac{c}{n})=+1}(i-\zeta_n^{c})(i-\zeta_n^{-c})
\\=&(-i)^{\frac{\phi(n)}{2}}\prod_{0<c<n,(\frac{c}{n})=+1}
\zeta_n^{-c}\prod_{0<c<n,(\frac{c}{n})=+1}\frac{1-\zeta_n^{4c}}{1-\zeta_n^{2c}}.
\end{align*}
By Lemma \ref{Lemma of Sn(1)} we have
$$\prod_{0<c<n,(\frac{c}{n})=+1}\frac{1-\zeta_n^{4c}}{1-\zeta_n^{2c}}=1.$$
And as in the proof of Lemma \ref{Lemma of Sn(1)} we have
$$\Arg(\prod\limits_{\substack{0<c<n,\\(\frac{c}{n})=+1}}\zeta_n^{-c})\equiv
\frac{2\pi}{n}\sum_{0<c<n/2}\(\frac{c}{n}\)c\equiv0\pmod{2\pi\Z}.$$
Hence $\Phi_n(i)=(-i)^{\frac{\phi(n)}{2}}=1$. Now we can show that $S_n(i)\in\Q(\sqrt{n})$ in this case.
Clearly the Galois group
$$\Gal\(\Q(i,\zeta_n)/\Q(\sqrt{n})\)=\{\tau_a: a\in(\Z/4n\Z)^{\times}\ \text{and}\ \(\frac{4n}{a}\)=+1\},$$
where $\tau_a$ is a $\Q$-automorphism of $\Q(\zeta_{4n})$
with $\tau_a(\zeta_{4n})=\zeta_{4n}^a$. For each $\tau_a\in\Gal(\Q(\zeta_{4n})/\Q(\sqrt{n}))$, it
is easy to see that
$(\frac{4n}{a})=+1$ implies that either $a\equiv 1\pmod 4$ and $(\frac{a}{n})=+1$ or $a\equiv3\pmod4$ and $(\frac{a}{n})=-1$. If $a\equiv 1\pmod 4$ and $(\frac{a}{p})=+1$, then clearly $\tau_a$ acts trivially on $S_n(i)$. Suppose now that $a\equiv3\pmod4$ and $(\frac{a}{n})=-1$. Then
\begin{equation*}
\tau_a(S_n(i))=\prod_{0<c<n,(\frac{c}{n})=+1}(-i-\zeta_n^{-c})=\frac{\Phi_n(i)}{S_n(-i)}=S_n(i).
\end{equation*}
The last equality follows from Lemma \ref{Lemma of Sn(i) times Sn(-i)} and (\ref{equation of Phi n (i)}).
This completes the proof.
\end{proof}
Now we are in the position to prove our main results.

\noindent{\bf Proof of Theorem\ \ref{Theorem A}}. By the class number formula of real quadratic field
(cf. \cite[p. 344, Theorem 2]{BS}) we have
\begin{equation}\label{equation of the class number of real quadratic field}
\ve_{4n}^{h(4n)}=\prod\limits_{\substack{0<b<2n\\ (\frac{4n}{b})=-1}}\sin\frac{\pi b}{4n}\ {\Big /}\prod\limits_{\substack{0<c<2n\\ (\frac{4n}{c})=+1}}\sin\frac{\pi c}{4n}
=\prod\limits_{\substack{0<c<2n\\ (\frac{4n}{c})=+1}}\cot\frac{\pi c}{4n}.
\end{equation}
The last equality of (\ref{equation of the class number of real quadratic field}) holds since for any $0<c<2n$,
$$\(\frac{4n}{c}\)=+1\ \text{if and only if}\ \(\frac{4n}{2n-c}\)=-1.$$
Clearly there exist $s_0,t_0\in\Z$ with $2\nmid t_0$ and $ns_0+4t_0=1$ such that
$\zeta_{4n}=i^{s_0}\zeta_n^{t_0}$. One can verify that
\begin{equation}\label{equation of phi n}
\#\{0<c<2n:\ \(\frac{4n}{c}\)=+1\}=\frac{\phi(n)}{2}.
\end{equation}
By (\ref{equation of the class number of real quadratic field}) and (\ref{equation of phi n}), we have
\begin{equation}\label{equation of zeta n and zeta 4n}
\ve_{4n}^{h(4n)}=(-i)^{\frac{\phi(n)}{2}}\prod\limits_{\substack{0<c<2n\\ (\frac{4n}{c})=+1}}\frac{1+\zeta_{4n}^c}{1-\zeta_{4n}^c}
=(-i)^{\frac{\phi(n)}{2}}\prod\limits_{\substack{0<c<2n\\ (\frac{4n}{c})=+1}}\frac{1-i^c\zeta_n^{t_0c}}{1+i^c\zeta_n^{t_0c}}.
\end{equation}
The last equality of (\ref{equation of zeta n and zeta 4n}) follows from $s_0\equiv 3\pmod 4$.

For each $0<c<2n$ the equality $(\frac{4n}{c})=+1$ implies
$$2\nmid c\ \text{and}\ \(\frac{c}{n}\)=(-1)^{\frac{c-1}{2}}.$$
We therefore get that
\begin{equation}\label{equation of c mod 4}
\ve_{4n}^{h(4n)}=(-i)^{\frac{\phi(n)}{2}}
\prod\limits_{\substack{0<c<n\\ c\equiv1\pmod4}}\frac{1-i\zeta_n^{t_0c(\frac{c}{n})}}{1+i\zeta_n^{t_0c(\frac{c}{n})}}
\prod\limits_{\substack{0<c<n\\ c\equiv3\pmod4}}\frac{1+i\zeta_n^{-t_0c(\frac{c}{n})}}{1-i\zeta_n^{-t_0c(\frac{c}{n})}}.
\end{equation}
Simplifying (\ref{equation of c mod 4}), we have
\begin{equation}\label{equation of Jacobi symbol of c}
\ve_{4n}^{h(4n)}=(-i)^{\frac{\phi(n)}{2}}(-1)^{\delta(4n)}
\prod\limits_{\substack{0<c<n,2\nmid c\\ (\frac{c}{n})=+1}}\frac{1-i\zeta_n^{t_0c}}{1+i\zeta_n^{t_0c}}
\prod\limits_{\substack{0<c<n,2\nmid c\\ (\frac{c}{n})=-1}}\frac{1+i\zeta_n^{t_0c}}{1-i\zeta_n^{t_0c}},
\end{equation}
where
\begin{equation}\label{equation of delta 4n}
\delta(4n)=\#\{0<c<n:\ \(\frac{4n}{c}\)=-1\}.
\end{equation}
Note that $(\frac{t_0}{n})=1$. And for each $0<c<n$, we have
$$2\nmid c\ \text{and}\ \(\frac{c}{n}\)=-1\ \text{if and only if}\ 2\mid n-c\ \text{and}\ \(\frac{n-c}{n}\)=+1.$$
Hence we further obtain
\begin{equation}\label{equation of Jacobi symbol of c =1}
\ve_{4n}^{h(4n)}=(-i)^{\frac{\phi(n)}{2}}(-1)^{\delta(4n)+\lambda(4n)}
\prod\limits_{\substack{0<c<n\\ (\frac{c}{n})=+1}}\frac{1-i\zeta_n^c}{1+i\zeta_n^c},
\end{equation}
where
\begin{equation}\label{lambda 4n}
\lambda(4n)=\#\{0<c<n/2:\ \(\frac{c}{n}\)=\(\frac{2}{n}\)\}.
\end{equation}
Observing that
$$\frac{1-i\zeta_n^c}{1+i\zeta_n^c}=-1\times\frac{-i-\zeta_n^c}{i-\zeta_n^c},$$
we finally obtain the following identity
\begin{equation}\label{equation of the square of Sn(i)}
\ve_{4n}^{h(4n)}=i^{\frac{\phi(n)}{2}}(-1)^{\delta(4n)+\lambda(4n)}\cdot\frac{S_n(-i)}{S_n(i)}
\end{equation}
By Lemma \ref{Lemma of Sn(i) times Sn(-i)} we have
\begin{equation}\label{equation of Sn(i)}
S_n(i)^2\ve_{4n}^{h(4n)}=i^{\frac{\phi(n)}{2}}\cdot(-1)^{\delta(4n)+\lambda(4n)}\cdot\theta(4n),
\end{equation}
where
$$\theta(4n)=\begin{cases}(-1)^{\frac{n+1}{4}}&\mbox{if}\ n\ \text{is prime},
\\1&\mbox{otherwise}.\end{cases}$$
It follows from (\ref{equation of Sn(i)}) that $S_n(i)$ is a unit in $\Z[\zeta_{4n}]$.
We now divide the remaining proof into two cases.

{\it Case}\ 1. $n$ is not a prime.

In this case, by Lemma \ref{Lemma of Galois} we have $S_n(i)\in\Z(\sqrt{n})$. Thus
$$S_n(i)^2\ve_{4n}^{h(4n)}=(-1)^{\frac{\phi(n)}{4}}\cdot(-1)^{\delta(4n)+\lambda(4n)}>0.$$
Hence we have
\begin{equation}\label{equation n is not a prime}
S_n(i)^2\ve_{4n}^{h(4n)}=1.
\end{equation}
We set $S_n(i)=a_n+b_n\sqrt{n}$ with $a_n,b_n\in\Z$. Since $S_n(i)$ is a unit, we get
$$a_n^2-nb_n^2=1.$$
Now we give the explicit value of $S_n(i)$. By (\ref{equation n is not a prime}) we have
$$S_n(i)=\pm\ve_{4n}^{-h(4n)/2}.$$
Hence it is enough to determine the sign of $S_n(i)$. Let
$\tau_{n+2}\in\Gal(\Q(\zeta_{4n})/\Q)$ with $\tau_{n+2}(\zeta_{4n})=\zeta_{4n}^{n+2}$. It is easy to see that
$$\tau_{n+2}(S_n(i))=\begin{cases}S_n(i)&\mbox{if}\ (\frac{2}{n})=+1,
\\\\1/S_n(i)&\mbox{if}\ (\frac{2}{n})=-1.\end{cases}$$
Hence
$$\sgn(\tau_{n+2}(S_n(i)))=\sgn(S_n(i)).$$
We therefore consider the argument of $\tau_{n+2}(S_n(i))$. Then we have
\begin{align*}
\tau_{n+2}(S_n(i))=&\prod_{0<c<n,(\frac{c}{n})=+1}(i-\zeta_n^{2c})
=i^{\frac{\phi(n)}{2}}\prod_{0<c<n,(\frac{c}{n})=+1}(1+i\zeta_n^{2c})
\\=&i^{\frac{\phi(n)}{2}}
\prod_{0<c<n,(\frac{c}{n})=+1}\zeta_{8n}^{n+8c}(\zeta_{8n}^{n+8c}+\zeta_{8n}^{-(n+8c)})
\\=&i^{\frac{\phi(n)}{2}}
\prod_{0<c<n,(\frac{c}{n})=+1}\zeta_{8n}^{n+8c}\sqrt{2}\(\cos\frac{2c\pi}{n}-\sin\frac{2c\pi}{n}\).
\end{align*}
Using the similar method as in the proof of Lemma \ref{Lemma of Sn(1)}, it is easy to verify that  $\Arg(\tau_{n+2}(S_n(i))) \mod{2\pi\Z}$ is equal to
\begin{equation*}
\frac{\phi(n)\pi}{8}+
\frac{\pi}{2}\sum\limits_{\substack{0<c<\frac{n}{2}\\ (\frac{c}{n})=+1}}\sgn\(1-\cot\frac{2c\pi}{n}\)
-\frac{\pi}{2}\sum\limits_{\substack{\frac{n}{2}<c<n\\ (\frac{c}{n})=+1}}\sgn\(1-\cot\frac{2c\pi}{n}\)\pmod{2\pi\Z}.
\end{equation*}
Using some elementary properties of $\cot(x)$ and noting that $8\mid\phi(n)$, it is easy to verify that
\begin{equation}\label{equation of the argument when n is not a prime}
\Arg(\tau_{n+2}(S_n(i)))\equiv \pi\(\frac{\phi(n)}{8}+\alpha(n)\)\pmod{2\pi\Z}.
\end{equation}
Hence we have
\begin{equation}\label{explicit value of Sn(i) when n is not a prime}
S_n(i)=(-1)^{\frac{\phi(n)}{8}+\alpha(n)}\cdot\ve_{4n}^{-h(4n)/2}.
\end{equation}

{\it Case}\ 2. $n=p$ is a prime.

In this case, we first show that
\begin{equation}\label{equation of the square of Sp(i)}
S_p(i)^2\ve_{4p}^{h(4p)}=i^{\frac{p+3}{2}}.
\end{equation}
To see this, by Lemma \ref{Lemma of Galois} we may set
\begin{equation}\label{equation N}
(i-(-1)^{\frac{p+1}{4}})S_p(i)=a_p+b_p\sqrt{p}
\end{equation}
with $a_p,b_p\in\Z.$ And we also set
\begin{equation}\label{equation O}
\ve_{4p}^{h(4p)}=(u_p+v_p\sqrt{p})^{h(4p)}=A_p+B_p\sqrt{p}
\end{equation}
with $A_p,B_p\in\Z$. Since $\ve_{4p}>1$, we have $u_p$, $v_p>0$. This implies $A_p>0$.
Via computation we also have
\begin{equation}\label{equation P}
(a_p+b_p\sqrt{p})^2=2(A_p-B_p\sqrt{p})\cdot (-1)^{\frac{p+5}{4}+\delta(4p)+\lambda(4p)}.
\end{equation}
Since $A_p>0$, we have
\begin{equation}\label{equation which will be used in the proof of corollary}
a_p^2+pb_p^2=2A_p,
\end{equation}
and
$$\delta(4p)+\lambda(4p)\equiv \frac{p+5}{4}\pmod 2.$$
From this and (\ref{equation of Sn(i)}) we see that (\ref{equation of the square of Sp(i)})
holds. Moreover, by (\ref{equation P}) we have
$$\Norm_{\Q(\sqrt{p})/\Q}(a_p+b_p\sqrt{p})^2=4.$$
From this we obtain that
$$a_p^2-pb_p^2=\epsilon\cdot2$$
for some $\epsilon\in\{\pm1\}.$ Since $\epsilon\cdot2$ is a quadratic residue modulo $p$, we must have
$$a_p^2-pb_p^2=\(\frac{2}{p}\)2.$$
Now we compute the explicit value of $S_p(i)$. Let $\tau_{p+2}\in\Gal(\Q(\zeta_{4p})/\Q)$
with $\tau_{p+2}(\zeta_{4p})=\zeta_{4p}^{p+2}$. Then it is easy to see that
$$\tau_{p+2}(S_p(i))=\begin{cases}S_p(i)&\mbox{if}\ (\frac{2}{p})=+1,
\\\frac{i}{S_p(i)}&\mbox{if}\ (\frac{2}{p})=-1.\end{cases}$$
From this we first consider the argument of $\tau_{p+2}(S_p(i))$. As
\begin{align*}
\tau_{p+2}(S_p(i))=&\prod_{0<c<p,(\frac{c}{p})=+1}(i-\zeta_p^{2c}),
\end{align*}
using the essentially same method in the Case 1, we immediately obtain that $\Arg(\tau_{p+2}(S_p(i)))\mod {2\pi\Z}$ is equal to
\begin{equation*}
\frac{p+7}{8}\pi+\frac{\pi}{2}
\sum\limits_{\substack{0<c<\frac{p}{2}\\ (\frac{c}{p})=+1}}\sgn\(1-\cot\frac{2c\pi}{p}\)
-\frac{\pi}{2}\sum\limits_{\substack{\frac{p}{2}<c<p\\ (\frac{c}{p})=+1}}\sgn\(1-\cot\frac{2c\pi}{p}\)\pmod{2\pi\Z}.
\end{equation*}
Via computation, we have
\begin{equation}\label{equation of the argument when n is a prime}
\Arg(\tau_{p+2}(S_p(i)))\equiv \frac{3p+5}{8}\pi+\alpha(p)\pi\pmod{2\pi\Z}.
\end{equation}
From the above we have
$$\Arg(S_p(i))\equiv\begin{cases}(3p+5)\pi/8+\alpha(p)\pi\pmod{2\pi\Z}&\mbox{if}\ (\frac{2}{p})=+1,
\\\\(-3p-1)\pi/8+\alpha(p)\pi\pmod{2\pi\Z}&\mbox{if}\ (\frac{2}{p})=-1.\end{cases}$$
By (\ref{equation of the square of Sp(i)}) and the above equality, we have
$$S_p(i)=\begin{cases}\zeta_8^{(3p+5)/2}\cdot(-1)^{\alpha(p)}\cdot\ve_{4p}^{-h(4p)/2}&\mbox{if}\ p\equiv7\pmod8,
\\\\\zeta_8^{(-3p-1)/2}\cdot(-1)^{\alpha(p)}\cdot\ve_{4p}^{-h(4p)/2}&\mbox{if}\ p\equiv3\pmod8.\end{cases}$$
We observe that
$$\alpha(p)=
\left\lfloor\frac{p}{8}\right\rfloor+\#\{\frac{p}{8}<c<\frac{3p}{8}: \(\frac{c}{p}\)=-1\}
=\left\lfloor\frac{p}{8}\right\rfloor+\beta(p)-1.$$
When $p\equiv7\pmod8$, we set $p=8k+7$. In this case we have
\begin{align*}
S_p(i)&=\zeta_8^{12(k+1)+1}\cdot(-1)^{k+\beta(p)-1}\cdot\ve_{4p}^{-h(4p)/2}
\\&=(-1)^{\beta(p)}\cdot\ve_{4p}^{-h(4p)/2}\cdot(1+i)/\sqrt{2}.
\end{align*}
When $p\equiv3\pmod8$, we let $p=8k+3$. In this case we have
\begin{align*}
S_p(i)&=\zeta_8^{-4(3k+1)-1}\cdot(-1)^{k+\beta(p)-1}\cdot\ve_{4p}^{-h(4p)/2}
\\&=(-1)^{\beta(p)}\cdot\ve_{4p}^{-h(4p)/2}\cdot(1-i)/\sqrt{2}.
\end{align*}
In view of the above, we complete the proof.\qed

Now we are in the position to prove Corollary \ref{Corollary A} which extends S. Chowla's result \cite{Chowla}.

\noindent{\bf Proof of Corollary\ \ref{Corollary A}}. Let notations be as in the proof of Theorem \ref{Theorem A}.
Recall that $\ve_{4p}=u_p+v_p\sqrt{p}>1$.
As $$u_p^2-pv_p^2=+1,$$
we may set $u_p\equiv \iota\pmod p$ for some $\iota\in\{\pm1\}$. It is well known that $h(4p)$ is odd
(readers may consult \cite{Brown} for more details on the class number of real quadratic field).
Thus we have
$$u_p\equiv u_p^{h(4p)}\equiv A_p\pmod p.$$
Hence by (\ref{equation which will be used in the proof of corollary}), we have
$$a_p^2\equiv 2A_p\equiv 2u_p\equiv 2\iota\pmod p.$$
This shows that $2\iota$ is a quadratic residue modulo $p$. Hence
$$u_p\equiv \(\frac{2}{p}\)=(-1)^{\frac{p+1}{4}}\pmod p.$$
This completes the proof.\qed

\noindent{\bf Proof of Corollary\ \ref{corollary of EAACC}}. Let notations be as in the proof of
Theorem \ref{Theorem A}. By (\ref{equation of the square of Sp(i)}) and (\ref{equation O}) we have
$$(a_p+b_p\sqrt{p})^2=2(u_p-v_p\sqrt{p})^{h(4p)}.$$
Hence we have
\begin{equation}\label{equation of corollary EAACC i}
2a_pb_p\sqrt{p}=-2v_p\sqrt{p}\sum_{k=0}^{\infty}(v_p\sqrt{p})^{2k}u_p^{h(4p)-2k-1}
\binom {h(4p)}{2k+1}.
\end{equation}
By Corollary \ref{Corollary A} we know that $u_p\equiv(-1)^{(p+1)/4}\pmod p$. It is also well known that $h(4p)<p$ and $2\nmid h(4p)$. From the above we have
\begin{equation}\label{equation of corollary EAACC ii}
a_pb_p\equiv-v_p\cdot h(4p)\pmod p.
\end{equation}
Clearly $p\nmid a_p$. Our desired result now follows from (\ref{equation of corollary EAACC ii}). \qed
\maketitle

\acknowledgment\ This research was supported by the National Natural Science Foundation of
China (Grant No. 11971222).

\end{document}